\documentclass[11pt]{amsart}
\usepackage{amsfonts}
\usepackage{ifthen}
\usepackage{amsthm}
\usepackage{amsmath}
\usepackage{graphicx}
\usepackage{amscd,amssymb,amsthm}

\title{On Tur\'an type inequalities for modified Bessel functions}

\author[\'Arp\'ad Baricz]{\'Arp\'ad Baricz}
\author[Saminathan Ponnusamy]{Saminathan Ponnusamy}
\address{Department of Economics, Babe\c{s}-Bolyai University,
Cluj-Napoca 400591, Romania} \email{bariczocsi@yahoo.com}
\address{Department of Mathematics, Indian Institute of Technology Madras, Chennai 600036,
India} \email{samy@iitm.ac.in}

\newtheorem{theorem}{Theorem}
\newtheorem{remark}{Remark}


\begin{document}
\maketitle


\begin{abstract}
In this note our aim is to point out that certain inequalities for
modified Bessel functions of the first and second kind, deduced
recently by Laforgia and Natalini, are in fact equivalent to the
corresponding Tur\'an type inequalities for these functions.
Moreover, we present some new Tur\'an type inequalities for the
aforementioned functions and we show that their product is
decreasing as a function of the order, which has application in the
study of stability of radially symmetric solutions in a generalized
FitzHugh-Nagumo equation in two spatial dimensions. At the end of
this note a conjecture is posed, which may be of interest for
further research.
\end{abstract}

\section{Some inequalities for modified Bessel functions}
Let us denote with $I_{\nu}$ and $K_{\nu}$ the modified Bessel
functions of the first and second kind, respectively. For
definitions, recurrence formulas and other properties of modified
Bessel functions of the first and second kind we refer to the
classical book of G.N. Watson \cite{watson}.

In 2007, motivated by a problem which arises in biophysics, Penfold
et al. \cite{penfold} proved that the product of the modified Bessel
functions of the first and second kind, i.e. $u\mapsto
P_{\nu}(u)=I_{\nu}(u)K_{\nu}(u),$ is strictly decreasing on
$(0,\infty)$ for all $\nu\geq 0.$ It is worth mentioning that this
result for $\nu=n\geq 1,$ a positive integer, was verified in 1950
by Phillips and Malin \cite{phillips}. In order to shorten the proof
due to Penfold et al. \cite{penfold}, recently the first author
\cite{bariczprod} pointed out that the Tur\'an type inequalities for
modified Bessel functions of the first and second kinds are in fact
equivalent to some known inequalities for the logarithmic
derivatives of the functions in the question. For reader's
convenience we recall here the historical facts for these Tur\'an
type inequalities (see \cite{bariczprod,bariczaust} for more
details). More precisely, in view of the recurrence relations
\begin{equation}\label{r1}I_{\nu-1}(u)=(\nu/u) I_{\nu}(u)+I'_{\nu}(u)\end{equation}
and $$I_{\nu+1}(u) = I'_{\nu}(u) - (\nu/u) I_{\nu}(u),$$ the Tur\'an
type inequality
\begin{equation}\label{t1}
I_{\nu-1}(u) I_{\nu+1}(u) - \left[I_{\nu}(u)\right]^2 < 0
\end{equation}
is equivalent to
\begin{equation}\label{b1}
{uI_{\nu}'(u)}/{I_{\nu}(u)}<\sqrt{u^2+\nu^2},
\end{equation}
where $\nu>-1$ and $u>0.$ To the best of authors knowledge the
Tur\'an type inequality \eqref{t1} for $\nu>-1$ was proved first in
1951 by Thiruvenkatachar and Nanjundiah \cite{tiru} and later in
1974 by Amos \cite[p. 243]{amos} for $\nu\geq0.$ In 1991 Joshi and
Bissu \cite{joshi} proved also \eqref{t1} for $\nu\geq0,$ while in
1994 Lorch \cite{lorch} proved that \eqref{t1} in fact holds for all
$\nu\geq -1/2$ and $u>0.$ Recently, the first author
\cite{bariczaust} reconsidered the proof of Joshi and Bissu
\cite{joshi} and pointed out that \eqref{t1} and \eqref{b1} hold
true for all $\nu>-1$ and the constant zero on the right-hand side
of \eqref{t1} is the best possible. It is worth to mention that in
fact the function $\nu\mapsto I_{\nu}(u)$ is log-concave on
$(-1,\infty)$ for each fixed $u>0,$ as it was pointed out by Baricz
\cite{baricz4}. Finally, we note that the inequality \eqref{b1} for
$\nu>0$ was first proved by Gronwall \cite{gronwall} in 1932,
motivated by a problem in wave mechanics. This inequality in 1950
appeared also in Phillips and Malin's paper \cite{phillips} for
$\nu=n\geq 1,$ a positive integer.

Similarly, by using the recurrence relations
\begin{equation}\label{r2}K_{\nu-1}(u)=-(\nu/u)K_{\nu}(u)-K'_{\nu}(u)\end{equation}
and \begin{equation}\label{r3}K_{\nu+1}(u) =
-K'_{\nu}(u)+(\nu/u)K_{\nu}(u),\end{equation} it is easy to prove
(see \cite{bariczprod,bariczaust}) that the Tur\'an type inequality
\begin{equation}\label{t2}
K_{\nu-1}(u)K_{\nu+1}(u)-\left[K_{\nu}(u)\right]^2>0
\end{equation}
is equivalent to
\begin{equation}\label{b2}
uK_{\nu}'(u)/K_{\nu}(u)<-\sqrt{u^2+\nu^2},
\end{equation}
where $\nu\in\mathbb{R}$ and $u>0.$ The Tur\'an type inequality
\eqref{t2} was proved in 1978 by Ismail and Muldoon \cite{ismail}
and recently by the author (see \cite{baricz4,bariczaust}) by using
different methods. The constant zero on the right-hand side of
\eqref{t2} is the best possible. Note that for $\nu>1/2$ the Tur\'an
type inequality \eqref{t2} appears also on Laforgia and Natalini's
paper \cite{laforgia}. It is worth also to mention here that Ismail
and Muldoon \cite{ismail}, by using the Nicholson formula concerning
the product of two modified Bessel functions of different order,
proved that \cite{ismail} the function $\nu\mapsto
K_{\nu+a}(u)/K_{\nu}(u)$ is increasing on $\mathbb{R}$ for each
fixed $u>0$ and $a>0.$ As Muldoon \cite{muldoon2} pointed out, this
implies that $\nu\mapsto K_{\nu}(u)$ is log-convex on $\mathbb{R}$
for each fixed $u>0.$ Recently, by using the H\"older-Rogers
inequality, the first author \cite{baricz4} pointed out that the
function $\nu\mapsto K_{\nu}(u)$ is in fact strictly log-convex on
$\mathbb{R}$ for each fixed $u>0.$ Finally, we note that the
inequality \eqref{b2} was proved for $\nu=n\geq 1,$ a positive
integer, by Phillips and Malin \cite{phillips} in 1950.

Recently, motivated by some applications in finite elasticity,
Laforgia and Natalini \cite{natalini} proved the following
inequalities
\begin{equation}\label{l1}
\frac{I_{\nu}(u)}{I_{\nu-1}(u)}>\frac{-\nu+\sqrt{u^2+\nu^2}}{u}
\end{equation}
and
\begin{equation}\label{l2}
\frac{K_{\nu}(u)}{K_{\nu-1}(u)}<\frac{\nu+\sqrt{u^2+\nu^2}}{u},
\end{equation}
where $u>0$ and $\nu\geq0$ in the first inequality, and $u>0$ and
$\nu\in\mathbb{R}$ in the second inequality. However, inequality
\eqref{l1} is not new, as far as we know it was proved first by Amos
\cite{amos} in 1974. It is important to note here that Laforgia and
Natalini in order to deduce \eqref{l1} and \eqref{l2} used the
Tur\'an type inequalities \eqref{t1} and \eqref{t2}. Moreover,
inequality \eqref{l1} improves a recently deduced inequality by the
first author \cite{bariczjmaa}, which is useful in the study of the
generalized Marcum $Q-$function applied frequently in radar signal
processing. See \cite{bariczjmaa,natalini} for more details and also
\cite{baricz1,baricz2,baricz3,baricz4,bariczaust,joshi,laforgia,lorch,tiru}
for more details on Tur\'an type inequalities.

Here our aim is to show that the inequalities \eqref{l1} and
\eqref{l2} are in fact equivalent to the corresponding Tur\'an type
inequalities \eqref{t1} and \eqref{t2}. To see this we just need to
rewrite the inequality \eqref{b1}, by using \eqref{r1}, in the
following form
$$\frac{I_{\nu-1}(u)}{I_{\nu}(u)}<\frac{\nu+\sqrt{u^2+\nu^2}}{u},$$
which is actually equivalent to \eqref{l1}. Thus, the inequalities
\eqref{t1}, \eqref{b1} and \eqref{l1} are equivalent and hold true
for all $\nu>-1$ and $u>0.$ Similarly, in view of \eqref{r2}, the
inequality \eqref{b2} is equivalent to
$$\frac{K_{\nu-1}(u)}{K_{\nu}(u)}>\frac{-\nu+\sqrt{u^2+\nu^2}}{u},$$
which is in fact equivalent to \eqref{l2}. Thus, the inequalities
\eqref{t2}, \eqref{b2} and \eqref{l2} are equivalent.

Moreover, by using the corresponding counterparts of the Tur\'an
type inequalities \eqref{t1} and \eqref{t2} (or inequalities
\eqref{b1} and \eqref{b2}) we can obtain in a similar way the
corresponding counterparts of inequalities \eqref{l1} and
\eqref{l2}. More precisely, consider the Tur\'an type inequalities
\begin{equation}\label{t3}
I_{\nu}^2(u)-I_{\nu-1}(u)I_{\nu+1}(u)<I_{\nu}^2(u)/(\nu+1)
\end{equation}
and
\begin{equation}\label{t4}
K_{\nu}^2(u)-K_{\nu-1}(u)K_{\nu+1}(u)>K_{\nu}^2(u)/(1-\nu).
\end{equation}
Inequality \eqref{t3}, which holds for all $\nu>-1$ and $u>0,$ was
proved using different arguments in 1951 by Thiruvenkatachar and
Nanjundiah \cite{tiru} for $\nu>-1$, in 1991 by Joshi and Bissu
\cite{joshi} for $\nu\geq0$, and recently by the first author
\cite{baricz3,bariczaust} for $\nu>-1$. Note that the constant
$1/(\nu+1)$ on the right-hand side of \eqref{t3} is the best
possible, as it was pointed out recently in \cite{bariczaust}. The
Tur\'an type inequality \eqref{t4}, which holds for all $\nu>1$ and
$u>0,$ was proved recently in \cite{bariczaust}. We note that the
constant $1/(1-\nu)$ on the right-hand side of \eqref{t4} is the
best possible.

Now, consider the corresponding counterparts of \eqref{b1} and
\eqref{b2}, namely
\begin{equation}\label{b3}
uI_{\nu}'(u)/I_{\nu}(u)>\sqrt{u^2\nu/(\nu+1)+\nu^2}
\end{equation}
and
\begin{equation}\label{b4}
uK_{\nu}'(u)/K_{\nu}(u)>-\sqrt{u^2\nu/(\nu-1)+\nu^2}.
\end{equation}
The inequality \eqref{b3} is equivalent to \eqref{t3} and holds true
for all $\nu>-1$ and $u>0,$ as it was pointed out in
\cite{bariczprod}. This inequality appears also on Phillips and
Malin's paper \cite{phillips} for $\nu=n\geq 1,$ a positive integer.
The inequality \eqref{b4} is valid for all $\nu>1$ and $u>0.$ For
$\nu=n>1,$ a positive integer, it was proved by Phillips and Malin
\cite{phillips}, and for $\nu>1$ real by the first author
\cite{bariczaust}. Note that \eqref{b4} is equivalent to \eqref{t4},
which can be verified by using the corresponding recurrence
relations for the modified Bessel function of the second kind,
mentioned above.

Now, our aim is to show that by using the same argument as in the
proof of \eqref{l1} and \eqref{l2} we can obtain easily the
counterparts of \eqref{l1} and \eqref{l2}. Namely, by using
\eqref{r1} we conclude that \eqref{t3} or \eqref{b3} is equivalent
to
\begin{equation}\label{l3}
\frac{I_{\nu}(u)}{I_{\nu-1}(u)}<\frac{-\nu+\sqrt{\frac{\nu}{\nu+1}u^2+\nu^2}}{\frac{\nu}{\nu+1}u},
\end{equation}
which holds for all $u>0$ and all $\nu>0.$ Similarly, in view
\eqref{r2} we conclude that \eqref{t4} or \eqref{b4} is equivalent
to
\begin{equation}\label{l4}
\frac{K_{\nu}(u)}{K_{\nu-1}(u)}>\frac{\nu+\sqrt{\frac{\nu}{\nu-1}u^2+\nu^2}}{\frac{\nu}{\nu-1}u},
\end{equation}
which holds for all $\nu>1$ and all $u>0.$ We note that inequalities
\eqref{l3} and \eqref{l4} in the above forms seem to be new.

\begin{remark}
{\em We would like to take the opportunity to point out some minor
errors, which we found in the papers \cite{bariczprod,bariczaust}.
In \cite{bariczprod} the first author claimed that \eqref{b4} is
equivalent to the Tur\'an type inequality \cite[Eq. 3.3]{bariczprod}
\begin{equation}\label{t5}
(\nu-1)K_{\nu-1}(u)K_{\nu+1}(u)-(2\nu-1)\left[K_{\nu}(u)\right]^2<0
\end{equation}
and conjectured that \eqref{t5} holds true for all real $\nu\geq0.$
Unfortunately, this claim is not true. We note that \eqref{b4} is
actually equivalent to Tur\'an type inequality \eqref{t4}. This can
be verified by using the recurrence relations \eqref{r2} and
\eqref{r3} and by using the fact that $K_{\nu}$ is decreasing on
$(0,\infty)$ for all $\nu\in\mathbb{R}.$ All the same \eqref{t5} is
valid for all $\nu\in[0,1]$ and $u>0,$ it follows from \eqref{t2},
as it was pointed out in \cite{bariczaust}. Moreover, in
\cite{bariczaust} the first displayed inequality after Theorem 3.1
should be
\begin{equation}\label{t6}
K_{\nu}^2(u)-K_{\nu-1}(u)K_{\nu+1}(u)<\frac{\nu}{1-\nu}K_{\nu}^2(u)
\end{equation}
just like before Theorem 3.1 in \cite{bariczaust}. This is an
equivalent form of \eqref{t5} and in view of the fact that the
constant zero is the best possible in \eqref{t2} we conclude that
for $\nu>1$ the inequality \eqref{t6} or \eqref{t5} does not hold.
Moreover, the claim in \cite{bariczaust} that \eqref{t4} improves
\eqref{t5}, when $\nu>1,$ is not true. In view of \eqref{t6}, there
is no connection between \eqref{t4} and \eqref{t5}.}

{\em We note that using the recurrence relation
$K_{\nu}(u)=K_{-\nu}(u),$ the inequality \eqref{t5} can be rewritten
as
\begin{equation}\label{t7}
(\nu+1)K_{\nu-1}(u)K_{\nu+1}(u)-(2\nu+1)\left[K_{\nu}(u)\right]^2>0.
\end{equation}
In \cite{bariczprod} we conjectured that \eqref{t7} holds true for
all $\nu\leq 0$ and $u>0.$ It is important to point out here that in
view of the above discussion on \eqref{t5} the inequality \eqref{t7}
holds true for $\nu\in[-1,0],$ but does not holds for $\nu\leq-1.$}

{\em Finally, we note that in the proof of Theorem 3.1 in
\cite{bariczaust} the explicit formulas for the function
$\phi_{\nu}:(0,\infty)\to\mathbb{R},$ defined by
$$\phi_{\nu}(u)=1-\frac{K_{\nu-1}(u)K_{\nu+1}(u)}{K_{\nu}^2(u)},$$
and its derivative should be written as
\begin{align*}\phi_{\nu}(u)&=-\frac{1}{u}\left[\frac{4}{\pi^2}\int_0^{\infty}\frac{u\gamma(t)dt}{u^2+t^2}+
\frac{4}{\pi^2}\int_0^{\infty}\frac{u(t^2-u^2)\gamma(t)dt}{(u^2+t^2)^2}\right]\\&=
-\frac{4}{\pi^2}\int_0^{\infty}\frac{2t^2\gamma(t)dt}{(u^2+t^2)^2},\end{align*}
and consequently
$$\phi'_{\nu}(u)=\frac{32}{\pi^2}\int_0^{\infty}\frac{ut^2\gamma(t)dt}{(u^2+t^2)^3}.$$
Fortunately, this mistake does not affect the proof of Theorem 3.1
in \cite{bariczaust}, the proof of Tur\'an type inequalities
\eqref{t2} and \eqref{t4} is correct with the above small
modifications.}
\end{remark}

\section{Tur\'an type inequalities for modified Bessel functions}
In this section we are going to present some new Tur\'an type
inequalities for modified Bessel functions of the first and second
kind. As we can see below these inequalities are consequences of
some more general results on these functions. Before we state and
prove the first main result of this paper let us recall some basics.
By definition, a function $f$ is said to be completely monotonic
(c.m.) on an interval $\Delta$ if $f$ has derivatives of all orders
on $\Delta$ which alternate successively in sign, that is,
$$
(-1)^nf^{(n)}(u)\geq0
$$
for all $u\in\Delta$ and for all $n\geq0.$ It is known that c.m.
functions play an eminent role in areas like probability theory,
numerical analysis, physics, and theory of special functions. An
interesting exposition of the main results on c.m. functions is
given in Widder's book \cite{widder} (see also \cite{miller}).

The following results complement the discussion given Section 1.
\begin{theorem}\label{th1}
Let $I_{\nu}$ and $K_{\nu}$ be the modified Bessel functions of the
first and second kind. Then the following assertions are true:
\begin{enumerate}
\item the function $\nu\mapsto I_{\sqrt{\nu}}(u)$ is log-convex on $(0,\infty)$ for all $u>0$ fixed;
\item the function $\nu\mapsto K_{\sqrt{\nu}}(u)$ is log-concave on $(0,\infty)$ for all $u>0$ fixed;
\item the function $\nu\mapsto K_{\sqrt{\nu}}(u)/K_{\sqrt{\nu}+1}(u)$
is decreasing on $(0,\infty)$ for all $u>0$ fixed;
\item the function $\nu\mapsto
I_{\sqrt{\nu}}(u)K_{\sqrt{\nu}}(u)$ is log-convex on $(0,\infty)$
for all $u>0$ fixed;
\item the function $\nu\mapsto
I_{\sqrt{\nu}}(u)/K_{\sqrt{\nu}}(u)$ is log-convex on $(0,\infty)$
for all $u>0$ fixed;
\item the function $\nu\mapsto
I_{\nu}(u)/K_{\nu}(u)$ is log-concave on $(-1,\infty)$ for all $u>0$
fixed.
\end{enumerate}
In particular, for all $\nu>0$ and $u>0,$ the following Tur\'an type
inequalities hold $$I_{\sqrt{\nu+1}}^2(u)\leq
I_{\sqrt{\nu}}(u)I_{\sqrt{\nu+2}}(u),$$ $$K_{\sqrt{\nu+1}}^2(u)\geq
K_{\sqrt{\nu}}(u)K_{\sqrt{\nu+2}}(u),$$
$$K_{\sqrt{\nu+1}}(u)K_{\sqrt{\nu}+1}(u)\leq
K_{\sqrt{\nu}}(u)K_{\sqrt{\nu+1}+1}(u),$$
$$I_{\sqrt{\nu+1}}^2(u)K_{\sqrt{\nu+1}}^2(u)\leq
I_{\sqrt{\nu}}(u)K_{\sqrt{\nu}}(u)I_{\sqrt{\nu+2}}(u)K_{\sqrt{\nu+2}}(u),$$
$$I_{\sqrt{\nu+1}}^2(u)/K_{\sqrt{\nu+1}}^2(u)\leq
\left[I_{\sqrt{\nu}}(u)/K_{\sqrt{\nu}}(u)\right]\left[I_{\sqrt{\nu+2}}(u)/K_{\sqrt{\nu+2}}(u)\right]$$
and
$$I_{{\nu+1}}^2(u)/K_{{\nu+1}}^2(u)\leq
\left[I_{{\nu}}(u)/K_{{\nu}}(u)\right]\left[I_{{\nu+2}}(u)/K_{{\nu+2}}(u)\right]$$
\end{theorem}

\begin{proof}[\bf Proof]
For the proof of (1) -- (4) we use the following results for
modified Bessel functions:
\begin{enumerate}
\item[(a)] the function $\nu^2\mapsto I_{\nu}(u)$ is
c.m. on $[0,\infty)$ for each $u>0$ fixed;
\item[(b)] the function $\nu^2\mapsto 1/K_{\nu}(u)$ is
c.m. on $[0,\infty)$ for each $u>0$ fixed;
\item[(c)] the function $\nu\mapsto K_{\sqrt{\nu}+\alpha}(u)/K_{\sqrt{\nu}+\alpha+n}(u)$ is c.m. on $[0,\infty)$ for all $u>0,$ $\alpha\geq0$ and
$n\in\{1,2,3,{\dots}\}$ fixed;
\item[(d)] the function $\nu^2\mapsto I_{\nu}(u)K_{\nu}(v)$ is
c.m. on $(0,\infty)$ for each $v\geq u>0$ fixed.
\end{enumerate}

Parts (a) and (d) were proved by Hartman and Watson \cite{hartman1},
part (b) is due to Hartman \cite{hartman} and part (c) was obtained
by Ismail \cite{isma}. Now, recall that if the function $f$ is c.m.
and the function $g$ is nonnegative with a c.m. derivative, then the
composite function $f\circ g$ is also c.m. (see
\cite{bochner,widder}). Thus, since $\nu\mapsto [\sqrt{\nu}]'$ is
c.m. on $(0,\infty),$ from parts (a) -- (d) we obtain the following
results:
\begin{enumerate}
\item[(i)] the function $\nu\mapsto I_{\sqrt{\nu}}(u)$ is
c.m. on $(0,\infty)$ for each $u>0$ fixed;
\item[(ii)] the function $\nu\mapsto 1/K_{\sqrt{\nu}}(u)$ is
c.m. on $(0,\infty)$ for each $u>0$ fixed;
\item[(iii)] the function $\nu\mapsto K_{\sqrt{\nu}}(u)/K_{\sqrt{\nu}+1}(u)$ is c.m. on $(0,\infty)$ for all $u>0$ fixed;
\item[(iv)] the function $\nu\mapsto I_{\sqrt{\nu}}(u)K_{\sqrt{\nu}}(u)$ is
c.m. on $(0,\infty)$ for each $u>0$ fixed.
\end{enumerate}
On the other hand, each nonnegative c.m. function is log-convex (see
\cite{widder}), and thus parts (i), (ii) and (iv) imply parts (1),
(2) and (4) of this theorem, while part (3) follows from (iii).

We note that actually part (3) can be proved directly by using the
formula (see \cite{gros,ismail2})
$$\frac{K_{\nu-1}(\sqrt{u})}{\sqrt{u}K_{\nu}(\sqrt{u})}=\frac{4}{\pi^2}\int_0^{\infty}\frac{\gamma_{\nu}(t)dt}{u+t^2},\
\ \ \mbox{where}\ \ \
\gamma_{\nu}(t)=\frac{t^{-1}}{J_{\nu}^2(t)+Y_{\nu}^2(t)}
$$
and $u,\nu>0.$ Here $J_{\nu}$ and $Y_{\nu}$ stand for the Bessel
function of the first and second kind, respectively. More precisely,
in view of the Nicholson formula \cite{watson}
$$J_{\nu}^2(u)+Y_{\nu}^2(u)=\frac{8}{\pi^2}\int_0^{\infty}K_0(2u\sinh t)\cosh(2\nu t)dt,$$
the function $\nu\mapsto\gamma_{\nu}(t)$ is decreasing on
$(0,\infty)$ and so is the function $\nu\mapsto
\gamma_{\sqrt{\nu}+1}(t)$ for all $t>0$ fixed. This in turn implies
that the function
$$\nu\mapsto\frac{K_{\sqrt{\nu}}(u)}{K_{\sqrt{\nu}+1}(u)}=\frac{4u}{\pi^2}\int_0^{\infty}\frac{\gamma_{\sqrt{\nu}+1}(t)dt}{u^2+t^2}$$
is decreasing on $(0,\infty)$ for all $u>0$ fixed.

Finally, we note that part (3) can be obtained also directly by
using the fact that the function $\nu\mapsto
K_{\nu+a}(u)/K_{\nu}(u)$ is increasing on $\mathbb{R}$ for each
fixed $u>0$ and $a>0$ (see \cite{ismail}).

(5) This follows easily from parts (1) and (2) of this theorem.

(6) Similarly, this follows from the facts (see \cite{baricz4}) that
$\nu\mapsto I_{\nu}(u)$ is log-concave on $(-1,\infty)$ and
$\nu\mapsto K_{\nu}(u)$ is log-convex on $\mathbb{R}$ for all $u>0$
fixed.
\end{proof}

\section{Monotonicity properties of the product of modified Bessel
functions}

Our second main result reads as follows.
\begin{theorem}\label{th2}
The function $\nu\mapsto I_{\nu}(u)K_{\nu}(u)$ is decreasing and
$(0,\infty)$ for all $u>0$ fixed. Moreover, for all $\nu\geq1/2$ and
$u>0$ we have
\begin{equation}\label{h2}
2I_{\nu}(u)K_{\nu}(u)\leq{I_{\nu-1}(u)K_{\nu-1}(u)+I_{\nu+1}(u)K_{\nu+1}(u)}.
\end{equation}
In particular, the sequence $\{I_n(u)K_n(u)\}_{n\geq1}$ is
decreasing and convex and for all $u>0$ we have
\begin{equation}\label{h1}
I_0(u)K_0(u)>I_1(u)K_1(u)>I_2(u)K_2(u)> \dots>
I_n(u)K_{n}(u)>\cdots.
\end{equation}
\end{theorem}
\begin{proof}[\bf Proof]
Recall that the function  $\nu^2\mapsto I_{\nu}(u)K_{\nu}(v)$ is
c.m. on $(0,\infty)$ for each $v\geq u>0$ fixed (see
\cite{hartman1}), and consequently the function $\nu\mapsto
I_{\sqrt{\nu}}(u)K_{\sqrt{\nu}}(u)$ is also c.m. on $(0,\infty)$ for
each $u>0$ fixed (see part (iv) in the proof of Theorem \ref{th1}).
In particular, the function $\nu\mapsto
I_{\sqrt{\nu}}(u)K_{\sqrt{\nu}}(u)$ is decreasing on $(0,\infty)$
for each $u>0$ fixed, and hence the function $\nu\mapsto
I_{\nu}(u)K_{\nu}(u)$ is also decreasing and $(0,\infty)$ for all
$u>0$ fixed. Thus, for all $u>0$ we have
$$I_1(u)K_1(u)>I_2(u)K_2(u)> \dots> I_n(u)K_{n}(u)>{\cdots}.$$

Now, we are going to prove \eqref{h2}. By using the notation
$P_{\nu}(u)=I_{\nu}(u)K_{\nu}(u),$ the inequality \eqref{h2} can be
rewritten as
\begin{equation}\label{h4}
2P_{\nu}(u)\leq P_{\nu-1}(u)+P_{\nu+1}(u).
\end{equation}
We note that for the function $P_{\nu}$ the following recurrence
formulas (see \cite{penfold}) are valid
\begin{equation}\label{h5}
2\nu P_{\nu}'(u)=u(P_{\nu+1}(u)-P_{\nu-1}(u))
\end{equation}
and
\begin{equation}\label{h6}
2\nu P_{\nu}''(u)=4\nu
P_{\nu}(u)-(2\nu-1)P_{\nu-1}(u)-(2\nu+1)P_{\nu+1}(u).
\end{equation}
On the other hand, owing to Hartman \cite{hartman2}, we know that
the function $u\mapsto uP_{\nu}(u)$ is concave on $(0,\infty)$ for
all $\nu>1/2.$ Since $u\mapsto 2uI_{1/2}(u)K_{1/2}(u)=1-e^{-2u}$ is
concave on $(0,\infty),$ we conclude that in fact the function
$u\mapsto uP_{\nu}(u)$ is concave on $(0,\infty)$ for all $\nu\geq
1/2.$ Consequently, for all $u>0$ and $\nu\geq1/2$ we have
$$uP_{\nu}''(u)\leq -2P_{\nu}'(u).$$
Now, combining this with \eqref{h5} and \eqref{h6} we obtain for all
$u>0$ and $\nu\geq 1/2$
$$2\nu\left[2P_{\nu}(u)-(P_{\nu-1}(u)+P_{\nu+1}(u))\right]\leq0,$$
which is equivalent to \eqref{h4}. Finally, by using \eqref{h2} for
$\nu=1$ we obtain $$P_1(u)-P_0(u)\leq P_2(u)-P_{1}(u)$$ and hence by
using the fact that the sequence $\{I_n(u)K_n(u)\}_{n\geq1}$ is
decreasing for each $u>0,$ it follows that the sequence
$\{I_n(u)K_n(u)\}_{n\geq0}$ is also decreasing for each $u>0,$ i.e.
\eqref{h1} is valid. With this the proof is complete.
\end{proof}

{\bf Comments and concluding remarks}

\begin{enumerate}

\item[\bf 1.] In \cite{h2,h1,h5} (see also \cite{h3}),
the authors study the existence, stability and interaction of
localized structures in a one-dimensional generalized
FitzHugh-Nagumo type model. Recently, van Heijster and Sandstede
\cite{h4} started to analyze the existence and stability of radially
symmetric solutions in the planar variant of this model. The product
of modified Bessel functions $I_{\nu}(u)K_{\nu}(u)$ discussed in
this note arises naturally in their stability analysis, and the
monotonicity condition \eqref{h1} is important to conclude
(in)stability of these radially symmetric solutions.

After we have finished the first draft of this paper, van Heijster
informed us that the monotonicity of $\nu\mapsto
I_{\nu}(u)K_{\nu}(u)$ is in fact an immediate consequence of the
integral formula \cite[p. 98]{magnus}
$$I_{\nu}(u)K_{\nu}(u)=\int_0^{\infty}I_{2\nu}(2u\sinh t)e^{-2u\cosh t}dt$$
and the monotonicity of $\nu\mapsto I_{\nu}(u).$ More precisely, due
to Cochran \cite{cochran} we know that $\nu\mapsto I_{\nu}(u)$ is
strictly decreasing on $[0,\infty)$ for all $u>0$ fixed. By using
this result and the above integral formula we conclude that in fact
$\nu\mapsto I_{\nu}(u)K_{\nu}(u)$ is strictly decreasing on
$[0,\infty)$ for all $u>0$ fixed. Thanks are due to van Heijster for
the above information.

\item[\bf 2.] Let us recall that the function $u\mapsto
P_{\nu}(u)=I_{\nu}(u)K_{\nu}(u)$ is strictly decreasing on
$(0,\infty)$ for all $\nu>-1.$ For, $\nu=n\geq1,$ a positive
integer, this was proved in 1950 by Phillips and Malin
\cite{phillips}, for $\nu\geq0$ real in 2007 by Penfold et al.
\cite{penfold}, and for $\nu\geq-1/2$ and $\nu>-1$ recently by the
first author \cite{bariczprod,bariczaust}. We note that by using
\eqref{h5} and the monotonicity of $\nu\mapsto P_{\nu}(u)$, proved
in Theorem \ref{th2}, we obtain immediately that $u\mapsto
P_{\nu}(u)$ is strictly decreasing on $(0,\infty)$ for all $\nu>1.$

Now, let us consider the function $u\mapsto 2uP_{\nu}(u).$ Hartman
and Watson \cite{hartman1} proved that this function is a continuous
cumulative distribution function on $(0,\infty)$ for all
$\nu\geq1/2.$ This function actually maps $(0,\infty)$ into $(0,1).$
Moreover, later Hartman \cite{hartman2} proved that $u\mapsto
2uP_{\nu}(u)$ is concave on $(0,\infty)$ for all $\nu>1/2,$ and as
we pointed out in the proof of Theorem \ref{th2}, this result is
valid also for $\nu=1/2.$ We would like to point out here that
actually the monotonicity of $u\mapsto P_{\nu}(u)$ for $\nu\geq1/2$
is almost trivial by using the aforementioned concavity result of
Hartman. To see this let us recall a particular form of the Pinelis
version of the monotone form of l'Hospital's rule (see \cite{pin}):

{\em Let $f,g:(a,b)\subset\mathbb{R}\to\mathbb{R}$ be differentiable
functions on $(a,b)$ with $g'(u)\neq0$ for all $u\in(a,b).$
Furthermore, suppose that $\lim_{u\to a}f(u)=\lim_{u\to a}g(u)=0$
and $f'/g'$ is decreasing on $(a,b).$ Then the ratio $f/g$ is
decreasing too on $(a,b).$}

Appealing to this result, since
$\lim_{u\to0}2uP_{\nu}(u)=\lim_{u\to0}u/\nu=0,$ for $\nu\geq1/2,$ to
prove the monotonicity of $u\mapsto P_{\nu}(u)=uP_{\nu}(u)/u$ it is
enough to show that $u\mapsto [uP_{\nu}(u)]'/u'$ is decreasing,
which is clearly true because $u\mapsto uP_{\nu}(u)$ is concave.

\item[\bf 3.] Finally, we note that the product $P_{\nu}$ is useful
also in other problems of applied mathematics. For example, in 1986
Cantrell \cite{cantrell} derived tight upper bounds for the function
$u\mapsto uI_{n+1}(u)K_{m+1}(u),$ in order to obtain suitable
truncation and transient errors in the computation of the
generalized Marcum $Q-$function.
\end{enumerate}

Motivated by Theorems \ref{th1} and \ref{th2} we conjecture the
following.

\subsection*{Conjecture} The function $\nu\mapsto
I_{\nu}(u)K_{\nu}(u)$ is log-convex on $(-1,\infty)$ for all $u>0$
fixed.

\subsection*{Acknowledgments} The research of \'A. Baricz was supported by the J\'anos Bolyai Research
Scholarship of the Hungarian Academy of Sciences and was completed
during this author's visit to the Indian Institute of Technology
Madras. The visit was supported by a partial travel grant from the
Commission on Development and Exchanges, International Mathematical
Union. The authors are indebted to Peter van Heijster for his useful
comments and suggestions, which enhanced this paper.

\end{document}